\documentclass{amsart}
\usepackage{amssymb,amsmath}
\usepackage{graphicx}
\usepackage{ dsfont }

\input xy
\xyoption{arc}
\xyoption{all}
\xyoption{knot}

\newtheorem{theorem}{Theorem}
\newtheorem{definition}{Definition}
\newtheorem{lemma}[theorem]{Lemma}
\newtheorem{proposition}[theorem]{Proposition}

\begin{document}

\title[]{Unknotting Unknots}
\author{A. Henrich}
\address{Seattle University\\
Seattle, WA 98122}
\email{henricha@seattleu.edu}
\author{L. Kauffman}
\address{University of Illinois, Chicago\\
Chicago, IL 60607}
\email{kauffman@uic.edu}

\date{\today}

\begin{abstract} A knot is an an embedding of a circle into three--dimensional space. We say that a knot is unknotted if there is an ambient isotopy of the embedding to a standard circle. By representing knots via planar diagrams, we discuss the problem of unknotting a knot diagram when we know that it is unknotted. This problem is surprisingly difficult, since it has been shown that knot diagrams may need to be made more complicated before they may be simplified. We do not yet know, however, how much more complicated they must get. We give an introduction to the work of Dynnikov who discovered the key use of arc--presentations to solve the problem of finding a way to detect the unknot directly from a diagram of the knot. Using Dynnikov's work, we show how to obtain a quadratic upper bound for the number of crossings that must be introduced into a sequence of unknotting moves. We also apply Dynnikov's results to find an upper bound for the number of moves required in an unknotting sequence.
\end{abstract}
\maketitle

\section{Introduction}
When one first delves into the theory of knots, one learns that knots are typically studied using their diagrams. The first question that arises when considering these knot diagrams is: how can we tell if two knot diagrams represent the same knot? Fortunately, we have a partial answer to this question. Two knot diagrams represent the same knot in $\mathbb{R}^3$ if and only if they can be related by the Reidemeister moves, pictured below. Reidemeister proved this theorem in the 1920's~ \cite{Reid}, and 
it is the underpinning of much of knot theory. For example, J. W. Alexander based the original definition of his celebrated polynomial on the Reidemeister moves~\cite{Alex}. 

\begin{figure}[h]
\begin{center}
\includegraphics[trim = 0mm 0mm 0mm 0mm, clip, height=1.5in]{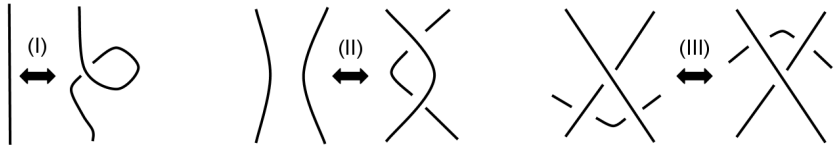}
\end{center}
\vspace{-.3in}
\caption{The three Reidemeister moves}
\end{figure}

Now imagine that you are presented with a complicated diagram of an unknot and you would like to use Reidemeister moves to reduce it to the trivial diagram that has no crossings. In considering a problem of this sort, you stumble upon a curious fact. Given a diagram of an unknot to be unknotted, it might be necessary to make the diagram more complicated before it can be simplified. We call such a diagram
a {\em hard unknot diagram} \cite{KL}.
A nice example of this is the Culprit, shown in Figure~\ref{culprit_intro}. If you look closely, you'll find that no simplifying type I or type II Reidemeister moves and no type III moves are available. Yet this is indeed the unknot. In order to unknot it, we need to introduce new crossings with Reidemeister I and II moves. In Figure~\ref{unknotCulprit}, we see that we can unknot the Culprit by making the diagram larger by two crossings (via a Reidemeister move of type two) and that it takes a total of ten Reidemeister moves to accomplish the unknotting.
\bigbreak

\begin{figure}[h]
\begin{center}
\includegraphics[height=1.2in, angle=90]{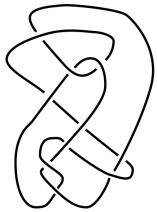}
\end{center}\label{culprit_intro}
\caption{The Culprit}
\end{figure}

\begin{figure}[h]
\begin{center}
\includegraphics[height=2in]{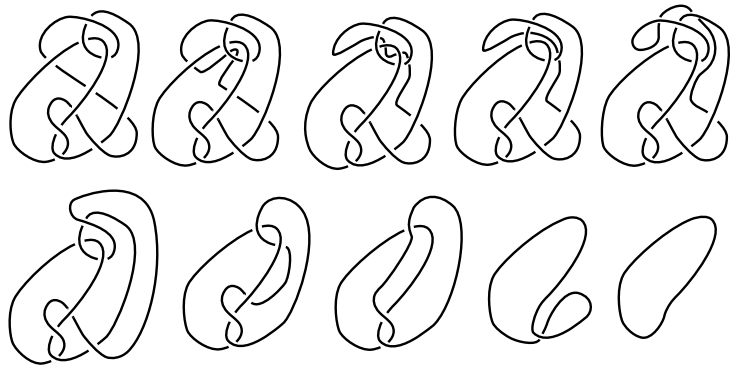}
\end{center}
\caption{The Culprit Undone}\label{unknotCulprit}
\end{figure}

\begin{figure}[h]
\begin{center}
\includegraphics[height=2in]{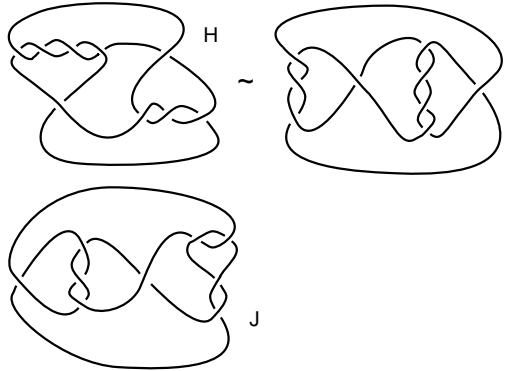}
\end{center}
\caption{The smallest hard unknots}
\end{figure}

\begin{figure}[h]
\begin{center}
\includegraphics[height=1in]{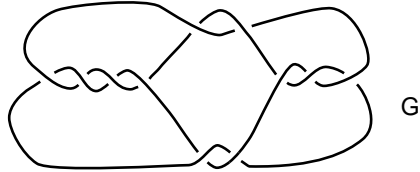}
\end{center}
\caption{The Goeritz unknot}
\end{figure}

In Figures 4 and 5 we indicate more examples of hard unknot diagrams. In Figure 4 we show the smallest possible such examples. In Figure 5 we show the very first such example, discovered by Goeritz  in 1934 ~\cite{Goeritz}. 
\bigbreak

At this point, we ask ourselves: how much more complicated does a diagram need to become before it can be simplified? Moreover, how many Reidemeister moves do we need to trivialize our picture? In this paper, we give a technique for finding upper bounds for these answers. In particular, we will prove the following theorem.
\bigbreak

\noindent {\bf Theorem 4.}
{\em Suppose $K$ is a diagram (in Morse form) of the unknot with crossing number $cr(K)$ and number of maxima $b(K)$. Let $M=2b(K)+cr(K)$. Then the diagram can be unknotted by a sequence of Reidemeister moves so that no intermediate diagram has more than $(M-2)^2$ crossings.}
\bigbreak

The definition of Morse form for a diagram will be given in the body of the paper.
In the case of our Culprit, we have  that $cr(K) = 10$ and $b(K) = 5. $ Thus $M = 20$ and 
$(M-2)^2 = 18^{2} = 324.$ In actuality we only needed a diagram with $12$ crossings in our unknotting
sequence. The theory of these bounds needs improvement, but it is, in fact, remarkable that there is a 
theory at all for such questions. Along with this theorem we will also give bounds on the number of 
Reidemeister moves needed for unknotting. We point the reader towards more results related to this question. As a disclaimer, we warn the reader that the difference between the lower bounds and upper bounds that are known is still vast. The quest for a satisfying answer to these questions continues.

\section{Preliminaries}

The method we present to find upper bounds makes use of a powerful result proven by Dynnikov in~\cite{dynnikov} regarding arc--presentations of knots. Here, we provide an overview of the theory of arc--presentations.  

\begin{definition}
An \emph{arc--presentation} of a knot is a knot diagram comprised of horizontal and vertical line segments such that at each crossing in the diagram, the horizontal arc passes under the vertical arc. Furthermore, we require that no two edges in an arc--diagram are colinear.

Two arc--presentations are \emph{combinatorially equivalent} if they are isotopic in the plane via an ambient isotopy of the form $h(x,y)=(f(x),g(y))$.

The \emph{complexity} $c(L)$ of an arc--presentation is the number of vertical arcs in the diagram.

We say more generally that a link diagram is {\em rectangular} if it has only vertical and horizontal edges.
In Figure~\ref{arc} we give an example of a rectangular diagram that is an arc--presentation and another example of a rectangular diagram that is not an arc--presentation.  
\end{definition}

Note that a rectangular diagram can naturally be drawn on a rectangular grid. If we start with such a grid and represent rectangular diagrams on the grid we have called these knots {\em mosaic knots} and 
used them to define a notion of {\em quantum knot.}  See ~\cite{Mosaic} for more about quantum knots.
For now, we focus our attention on arc--presentations.

\begin{figure}[h]
\begin{center}
\includegraphics[height=1.5in]{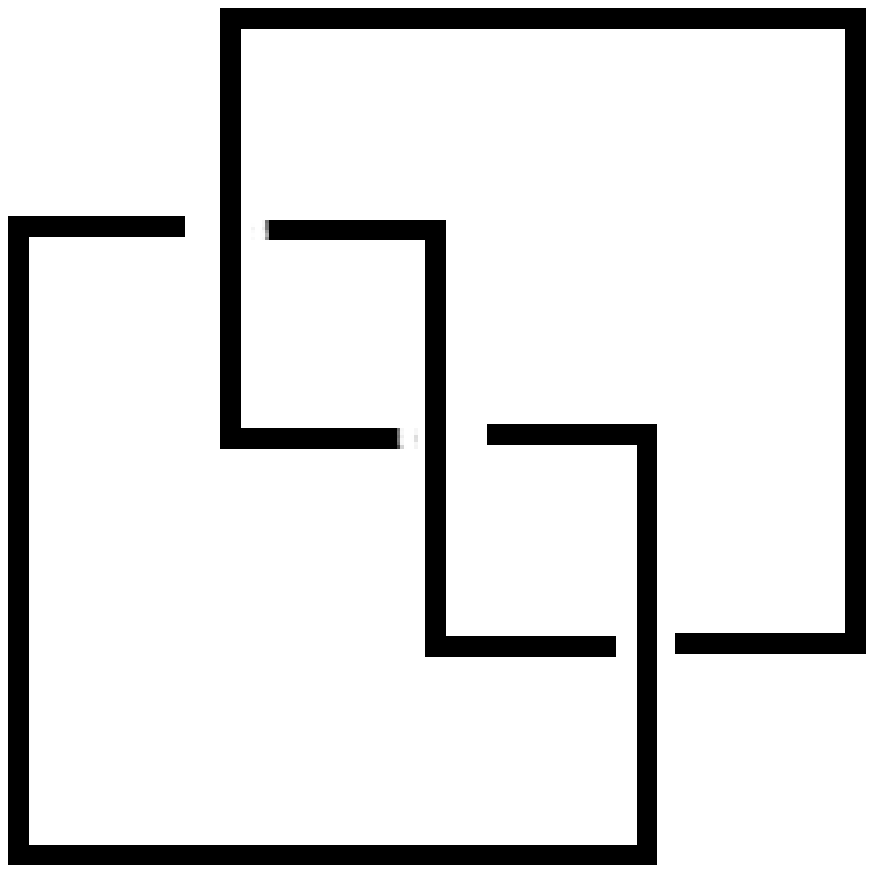}\includegraphics[height=1.5in]{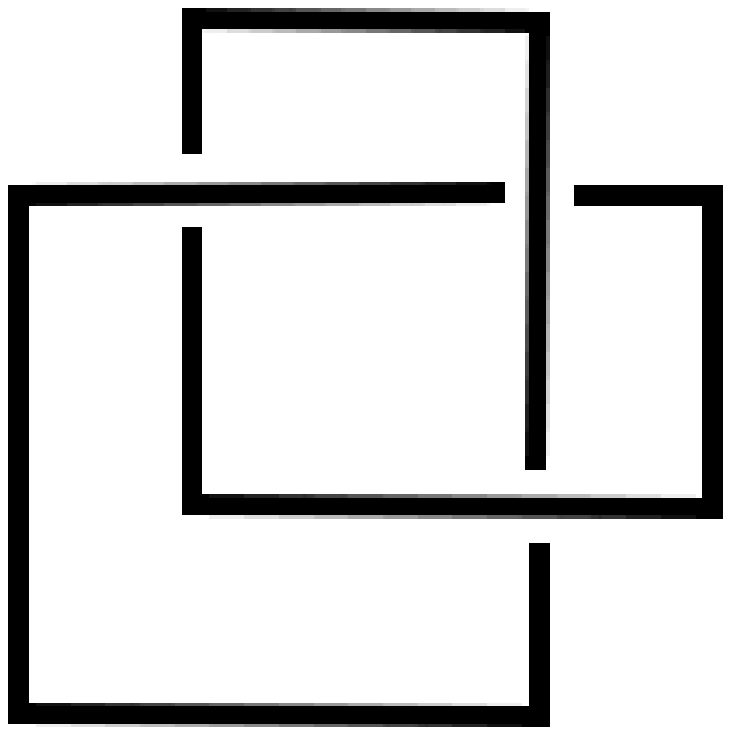}
\end{center}
\vspace{-.3in}
\caption{ \small{The picture on the left is an example of an arc--presentation of a trefoil. The picture on the right is an example that is \emph{not} an arc--presentation (since not all horizontal arcs pass under vertical arcs).}}
\label{arc}
\end{figure}

\begin{proposition}[Dynnikov]
Every knot has an arc--presentation. Any two arc--presentations of the same knot can be related to each other by a finite sequence of \emph{elementary moves}, pictured in Figures~\ref{stab} and~\ref{exch}.
\end{proposition}

The proof of this proposition is elementary, based on the Reidemeister moves. A sketch is provided in~\cite{dynnikov}. We will show how to convert a usual knot diagram to an arc--presentation in the next few paragraphs, making use of the concept of Morse diagrams of knots.

\begin{figure}[h]
\begin{center}
\includegraphics[trim = 0mm 30mm 0mm 30mm, clip, height=1in]{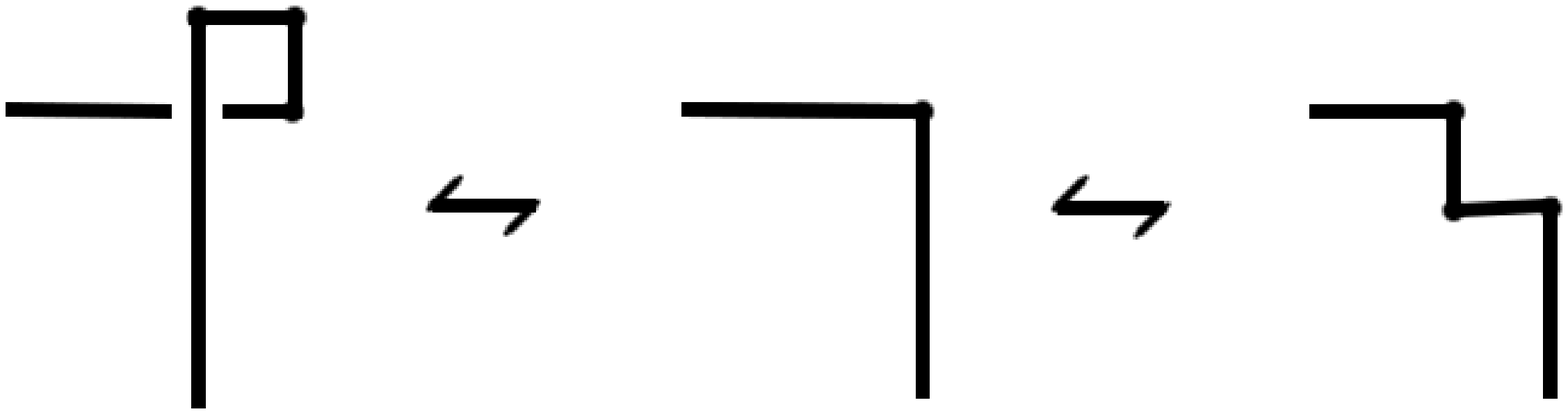}
\includegraphics[trim = 0mm 30mm 0mm 30mm, clip, height=1in]{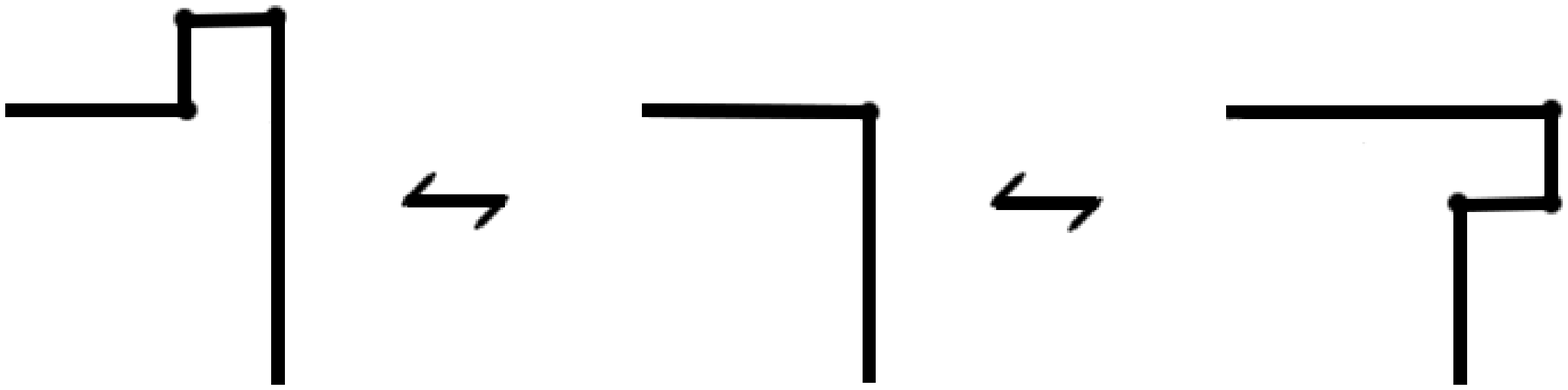}
\end{center}
\caption{ \small{Elementary (de)stabilization moves. Stabilization moves increase the complexity of the arc--presentation while destabilization moves decrease the complexity.}}
\label{stab}
\end{figure}

\begin{figure}[h]
\begin{center}
\includegraphics[trim = 0mm 30mm 0mm 30mm, clip, height=.9in]{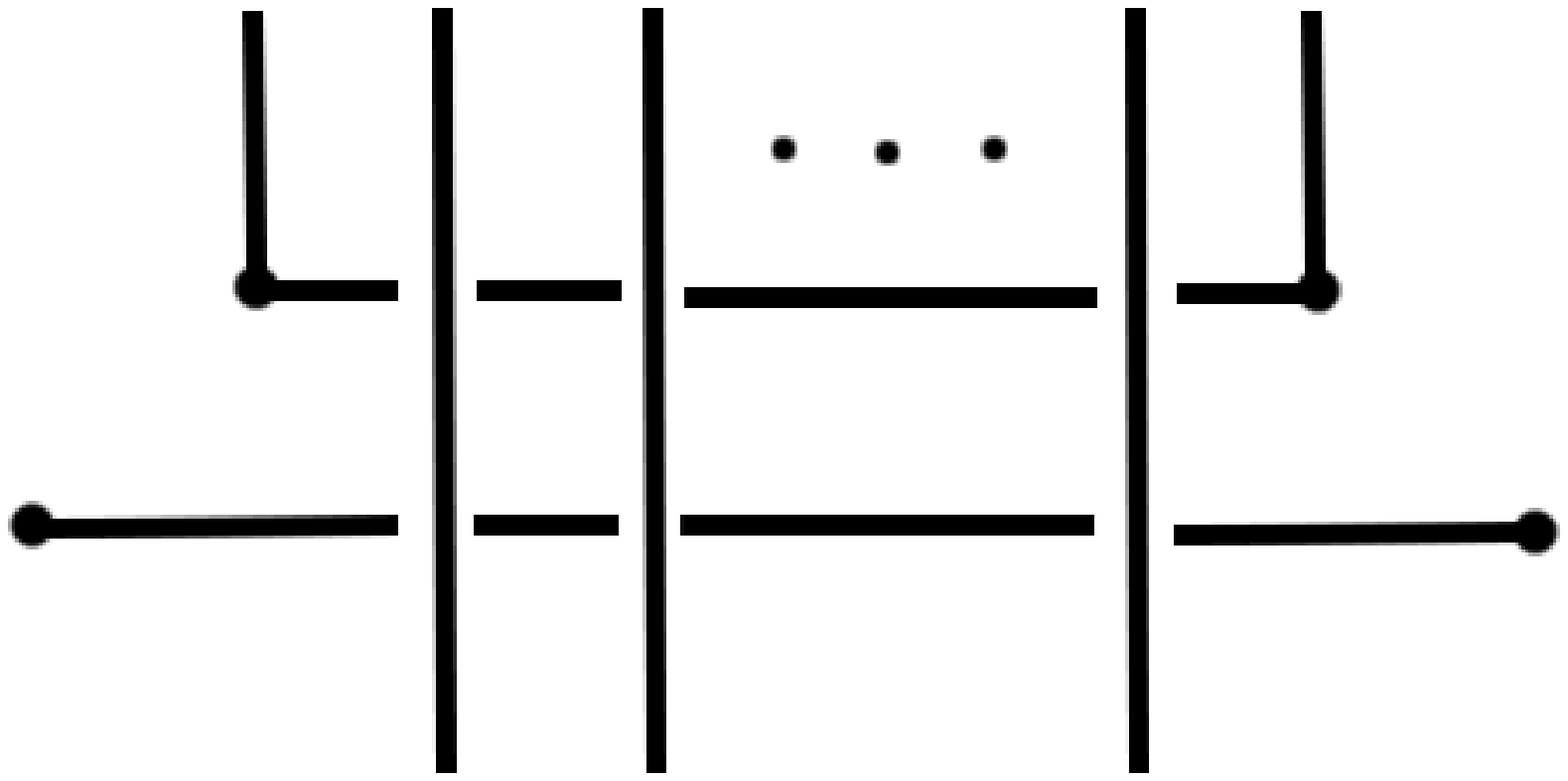}\includegraphics[trim = 0mm 30mm 0mm 30mm, clip, height=.9in]{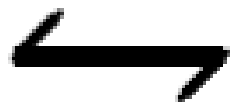}\includegraphics[trim = 0mm 30mm 0mm 30mm, clip, height=.9in]{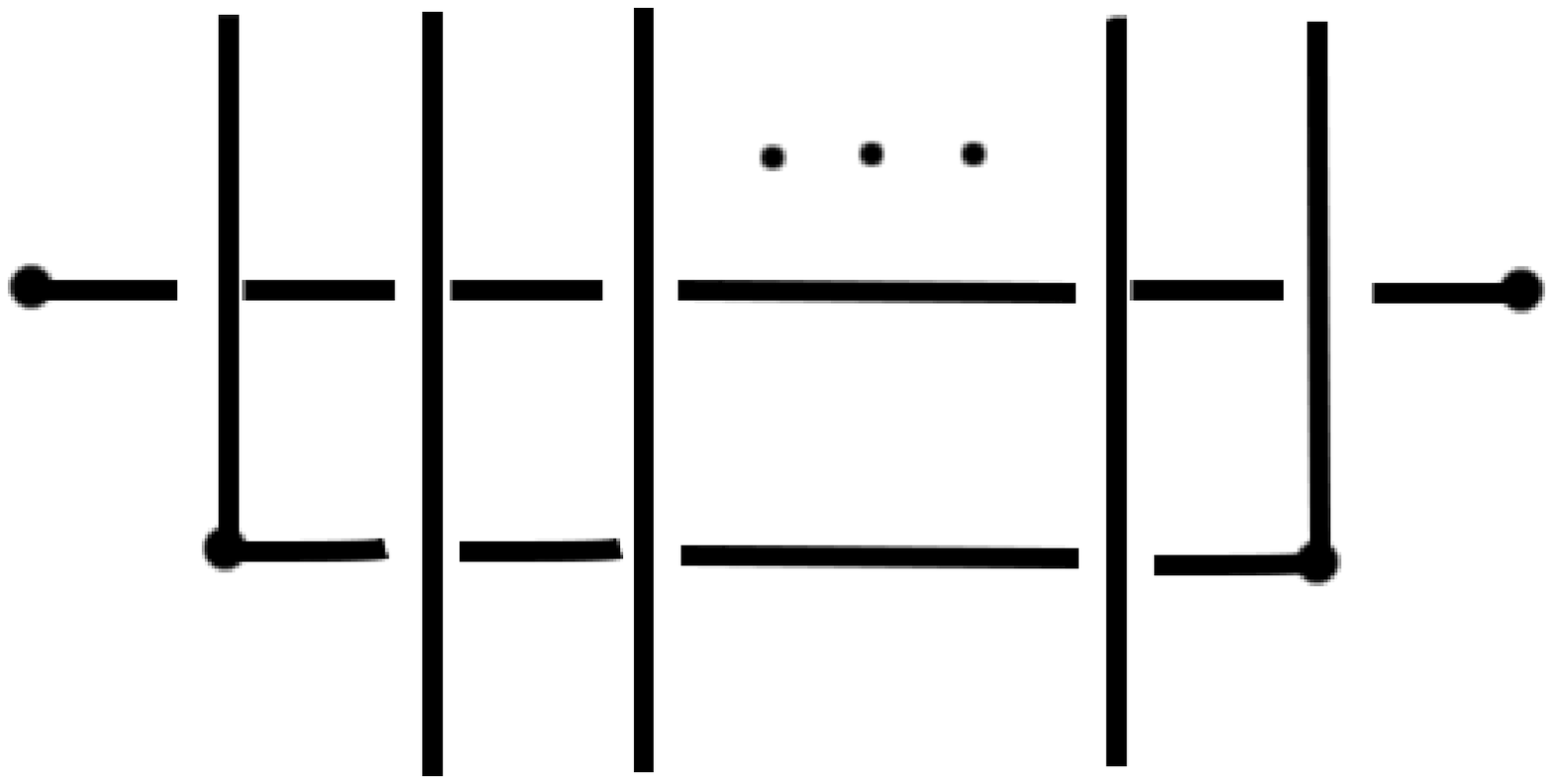}
\end{center}

\begin{center}
\includegraphics[trim = 0mm 30mm 0mm 30mm, clip, height=.9in]{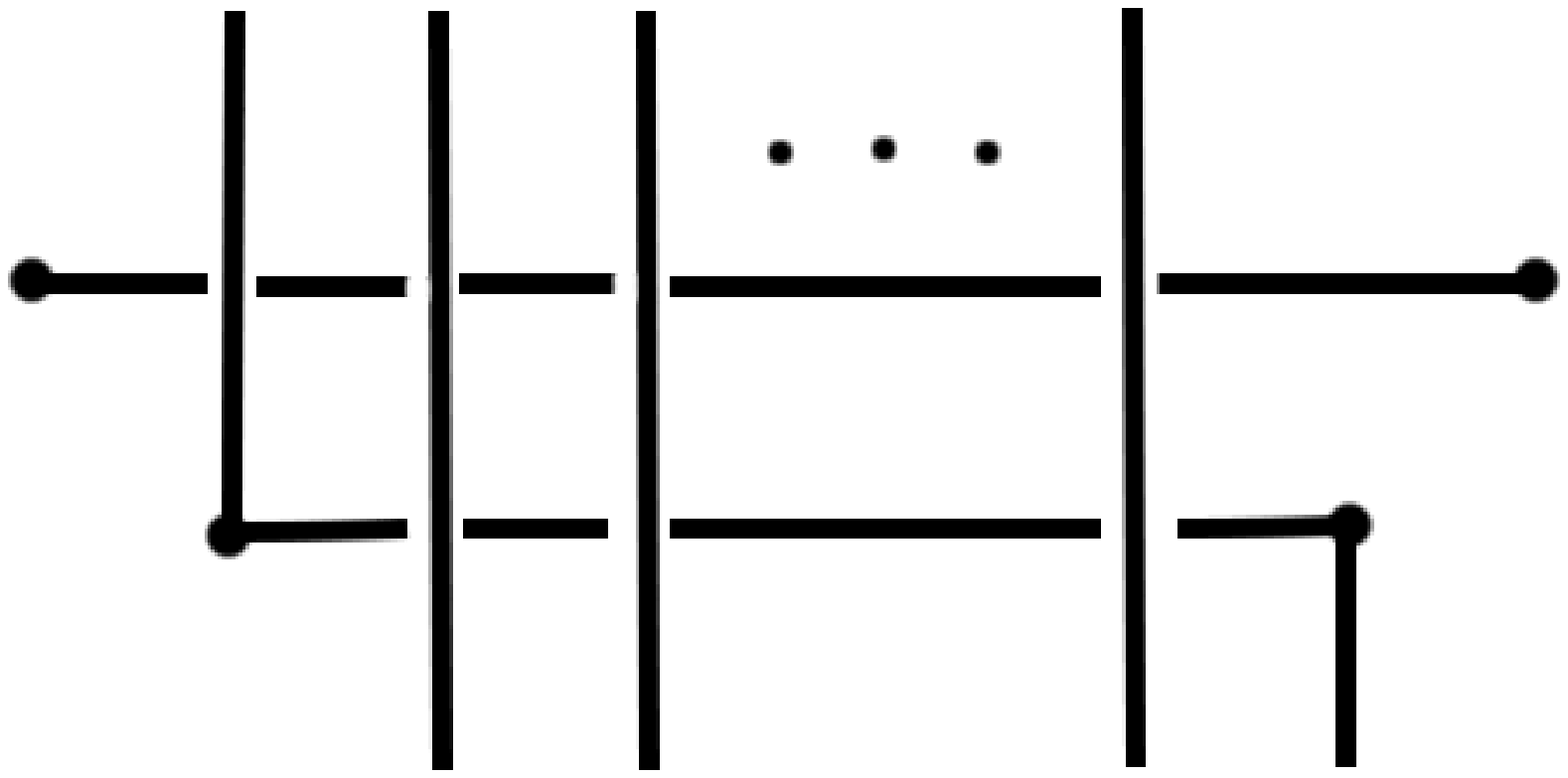}\includegraphics[trim = 0mm 30mm 0mm 30mm, clip, height=.9in]{arrow.eps}\includegraphics[trim = 0mm 30mm 0mm 30mm, clip, height=.9in]{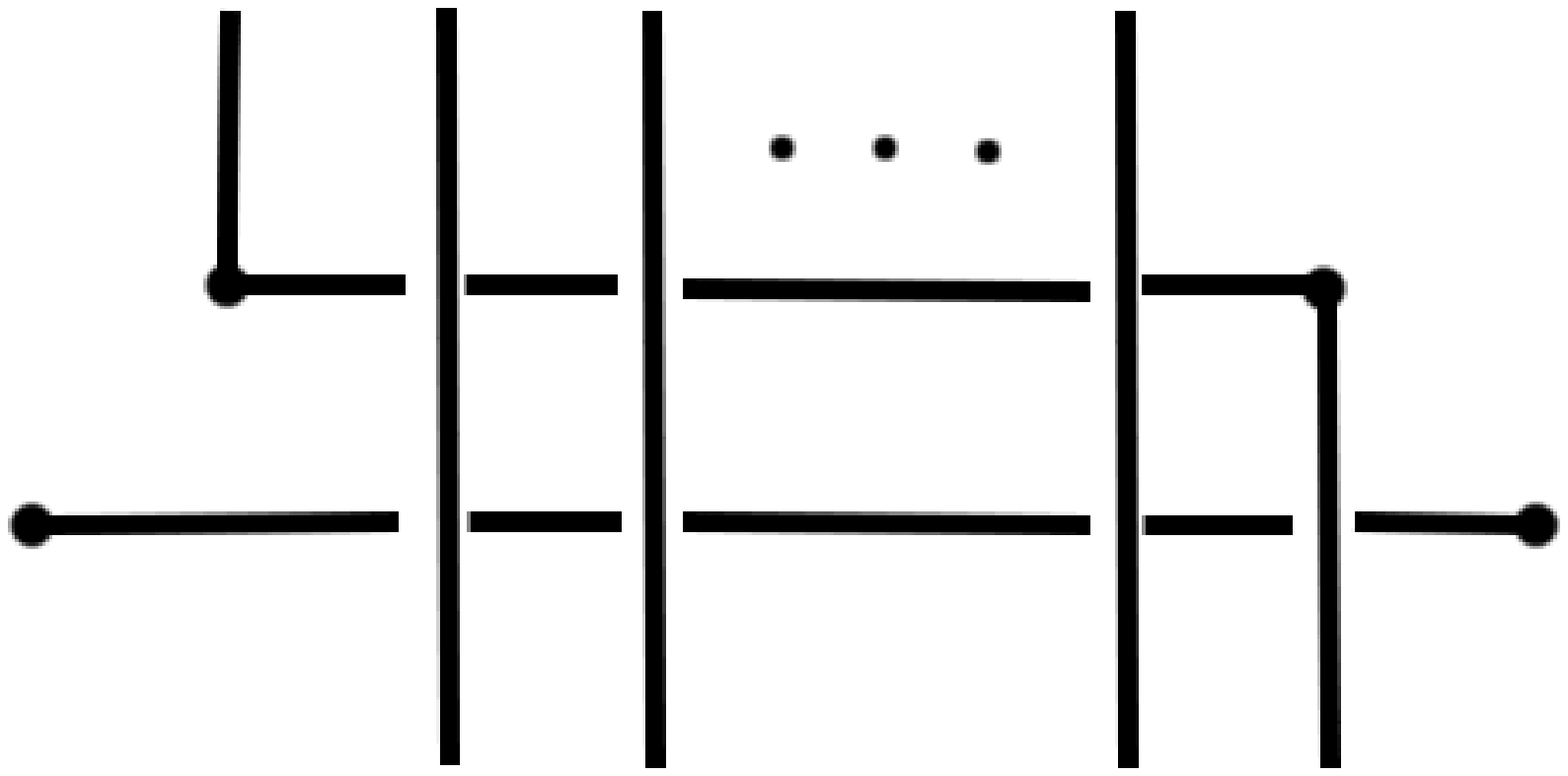}
\end{center}
\caption{ \small{Some examples of exchange moves. Other allowed exchange moves include switching the heights of two horizontal arcs that lie in distinct halves of the diagram.}}
\label{exch}
\end{figure}

\begin{definition} A knot diagram is in \emph{Morse form} if it has

\begin{enumerate}
\item no horizontal lines, 
\item no inflection points,
\item a single singularity at each height, and 
\item each crossing is oriented to create a 45 degree angle with the vertical axis.
\end{enumerate}
\end{definition}

We note that converting an arbitrary knot diagram into a diagram in Morse form requires no Reidemeister moves, only ambient isotopies of the plane. More information about Morse diagrams can be found in~\cite{lou}.

\begin{lemma}\label{arclemma}
Suppose a knot (or link) diagram $K$ in Morse form has $cr(K)$ crossings and $b(K)$ maxima. Then there is an arc--presentation $L_K$ of $K$ with complexity $c(L_K)$ at most $2b(K)+cr(K)$ that can be obtained by ambient isotopies of the plane (without the use of Reidemeister moves).
\end{lemma}
\begin{proof}
We begin with a diagram in Morse form and convert this diagram into a piecewise linear diagram composed of lines with slope $\pm1$ with a vertex corresponding to each maximum and minimum (possibly with additional vertices---at most one for each pair of successive extrema). If we rotate this diagram by 45 degrees, we have a diagram composed entirely of horizontal and vertical arcs with complexity at most $2b(K)$. 

This diagram may fail to be an arc--presentation of $K$ if any crossing has a horizontal overpass. If more than half of the crossings in $K$ have horizontal overpasses, we rotate the diagram by 90 degrees. Now, at least half of the crossings are in the proper form. Any remaining crossings containing a horizontal overpass may locally be rotated 90 degrees to form our arc--presentation $L_K$, as shown in Figure~\ref{rotation}. For each crossing that requires this move, the complexity of the rectangular diagram increases by at most 2. Thus, the overall complexity of our diagram increases by at most $2(\frac{1}{2}cr(K))=cr(K)$. It follows that $c(L_K)\leq 2b(K)+cr(K)$.  

Note that neither converting a Morse diagram into a piecewise linear diagram nor locally rotating a crossing uses Reidemeister moves. These are ambient isotopies of the plane.
\end{proof}

\begin{figure}[h]
\begin{center}
\includegraphics[trim = 0mm 20mm 0mm 20mm, clip,height=1in]{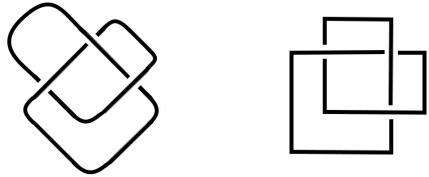}\includegraphics[trim = 0mm 20mm 0mm 20mm, clip,height=1in]{notarc.eps}
\end{center}
\caption{ \small{A Morse diagram of a knot and a corresponding rectangular diagram.}}
\label{morse}
\end{figure}

\begin{figure}[h]
\begin{center}
\includegraphics[trim = 0mm 30mm 0mm 30mm, clip, height=.7in]{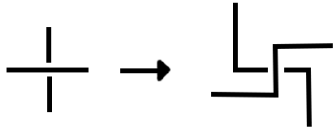}
\end{center}
\caption{ \small{Rotating a crossing to convert a rectangular diagram into an arc--presentation.}}
\label{rotation}
\end{figure}

\begin{figure}[h]
\begin{center}
\includegraphics[height=1.3in]{notarc.eps}\includegraphics[height=1.3in]{arrow.eps}\includegraphics[height=1.3in, angle=90]{notarc.eps}\includegraphics[height=1.3in]{arrow.eps}\includegraphics[height=1.3in]{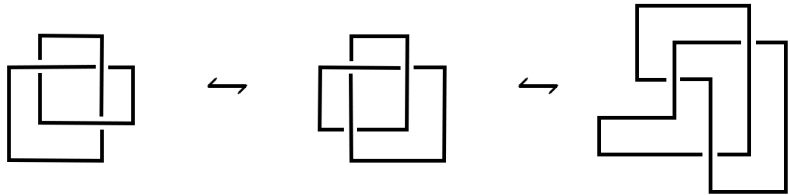}
\end{center}
\caption{ \small{Converting a rectangular diagram into an arc--presentation by rotating the diagram then rotating a crossing. Note that the resulting diagram can be reduced to the arc--presentation from Figure 1 with an exchange move that doesn't require any Reidemeister moves.}}
\label{morse}
\end{figure}

\section{Bounds on Crossings Needed to Simplify the Unknot}

Our motivation for using Dynnikov's work to find upper bounds for an unknotting Reidemeister sequence began with the following theorem from~\cite{dynnikov}.

\begin{theorem}[Dynnikov]\label{triv}
If $L$ is an arc--presentation of the unknot, then there exists a finite sequence of exchange and destabilization moves $$L\rightarrow L_1\rightarrow L_2\rightarrow
\cdots\rightarrow L_m$$ such than $L_m$ is trivial.\end{theorem}

What is particularly interesting about this result is that the unknot can be simplified \textbf{without increasing the complexity of the arc--presentation}, that is, without the use of stabilization moves. 
This gives a useful physical bound on how large a diagram can be. Furthermore, if we apply Dynnikov's
method to a knotted knot, it will stop on a diagram that is not a planar circle. Thus Dynnikov can
detect the unknot. 
\bigbreak

The problem of detecting the unknot has been investigated by many people. 
For example, the papers by Birman and Hirsch ~\cite{Alg}  and Birman and Moody ~\cite{Obstr} give such methods. More recently it has been shown that Heegard Floer Homology (a generalization of 
the Alexander polynomial) not only detects the unknot, but can be used to calculate the least genus of
an orientable spanning surface for any knot. This is an outstanding result and we recommend that the 
reader examine the paper by Manolescu, Oszvath, Szabo and Thurston ~\cite{Heegard} for more information. In that work, the Heegard Floer homology is expressed via a chain complex that is associated to a rectangular  diagram of just the type that Dynnikov uses. 
\bigbreak

Returning to the task at hand, we immediately derive a quadratic upper bound on the crossing number of diagrams in an unknotting sequence. A similar result can be found in~\cite{dynnikov}.

\begin{theorem}
Suppose $K$ is a diagram (in Morse form) of the unknot with crossing number $cr(K)$ and number of maxima $b(K)$. Then, for every $i$, the crossing number $cr(K_i)$ is no more than $(M-2)^2$ where $M=2b(K)+cr(K)$ and $K=K_0, K_1, K_2, ..., K_N$ is a sequence of knot diagrams such that $K_{i+1}$ is obtained from $K_i$ by a single Reidemeister move and $K_N$ is a trivial diagram of the unknot.
\end{theorem}

\begin{proof}
To begin, we notice that $K$ can be viewed as an arc--presentation of complexity $M$ by a simple ambient isotopy of the plane, as shown in Lemma~\ref{arclemma}. By Theorem~\ref{triv}, there is a sequence of arc--presentations beginning with $K$ and ending with the trivial arc--presentation each having complexity no more than $M$ such that a diagram and its successor are related by an exchange or a destabilization move. Each destabilization move either preserves or reduces the number of crossings in the diagram. In the case that a destabilization move reduces the number of crossings, it can be viewed as a simplifying Redemeister I move. Otherwise, it can be viewed as a simple ambient isotopy of the plane. 

When an exchange move is performed, on the other hand, its analogous Reidemeister sequence may require type II and type III Reidemeister moves. (See Figure~\ref{factor}.) At most one type II Reidemeister move is required for any given exchange move, so an exchange move factors through a Reidemeister sequence of moves that adds at most two crossings (since type III moves preserve the crossing number). However, it is important to note that a Reidemeister II move is needed if and only if the exchange move itself increases the number of crossings by two in the arc--presentation. Thus, no more crossings are added when factoring an exchange move through a Reidemeister sequence than are added in the exchange move itself.

\begin{figure}[h]
\begin{center}
\includegraphics[height=1.3in]{exch2.eps}\includegraphics[height=1.2in]{arrow.eps}\includegraphics[height=1.3in]{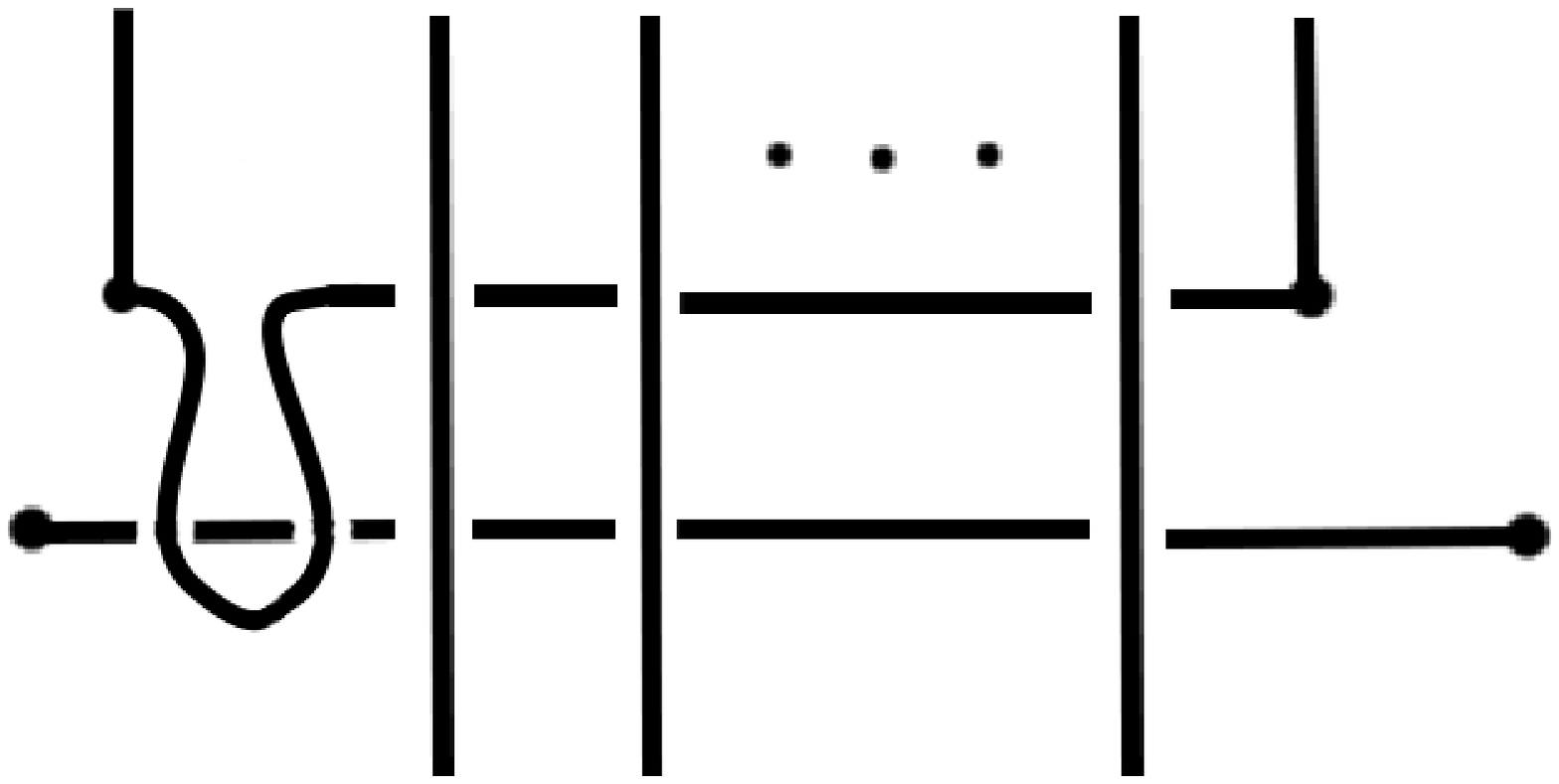}
\includegraphics[height=1.2in]{arrow.eps}
\vspace{-.3in}
\includegraphics[height=1.3in]{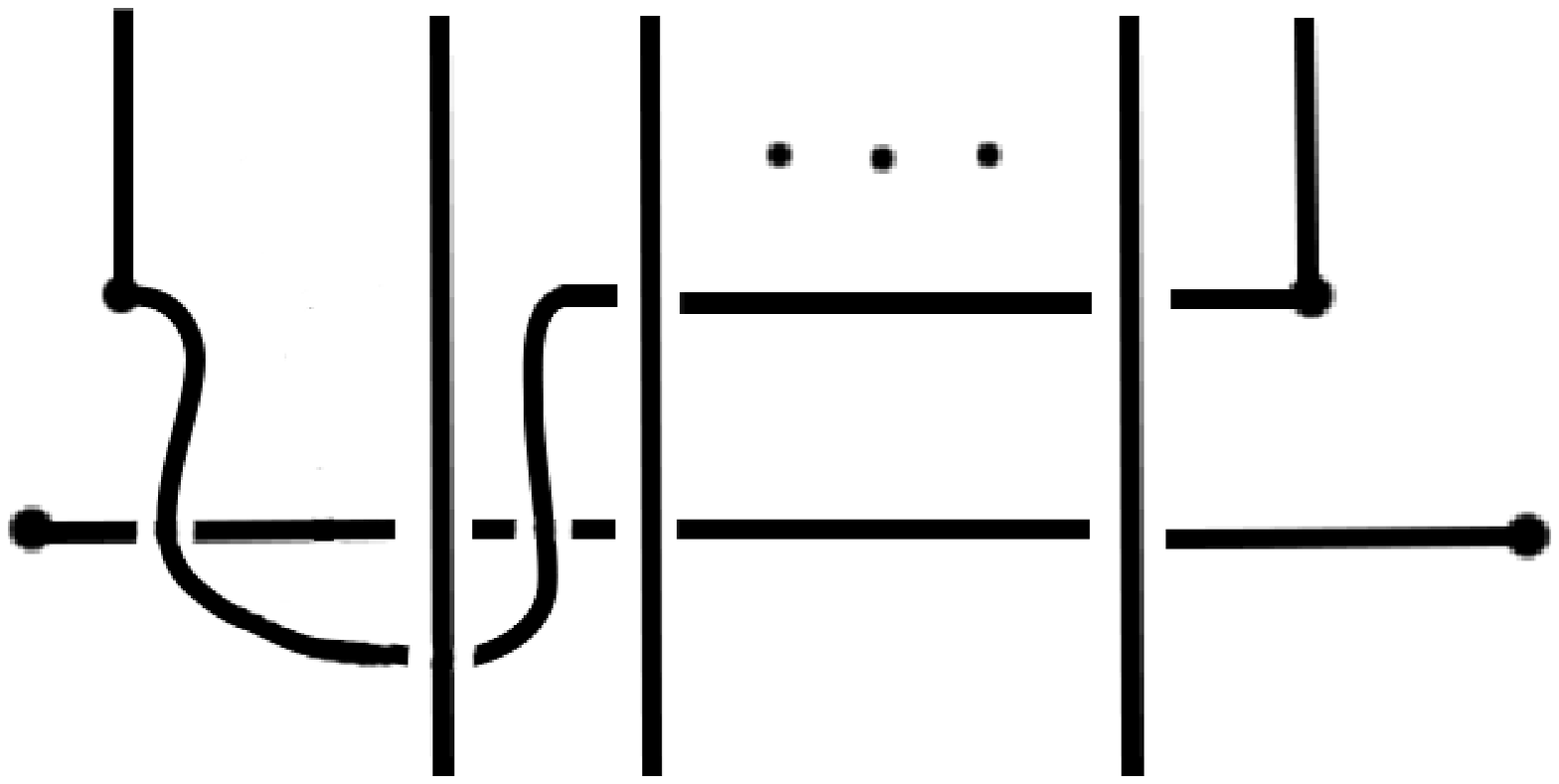}\includegraphics[height=1.2in]{arrow.eps}\includegraphics[height=1.2in]{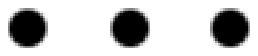}\includegraphics[height=1.2in]{arrow.eps}\includegraphics[height=1.4in]{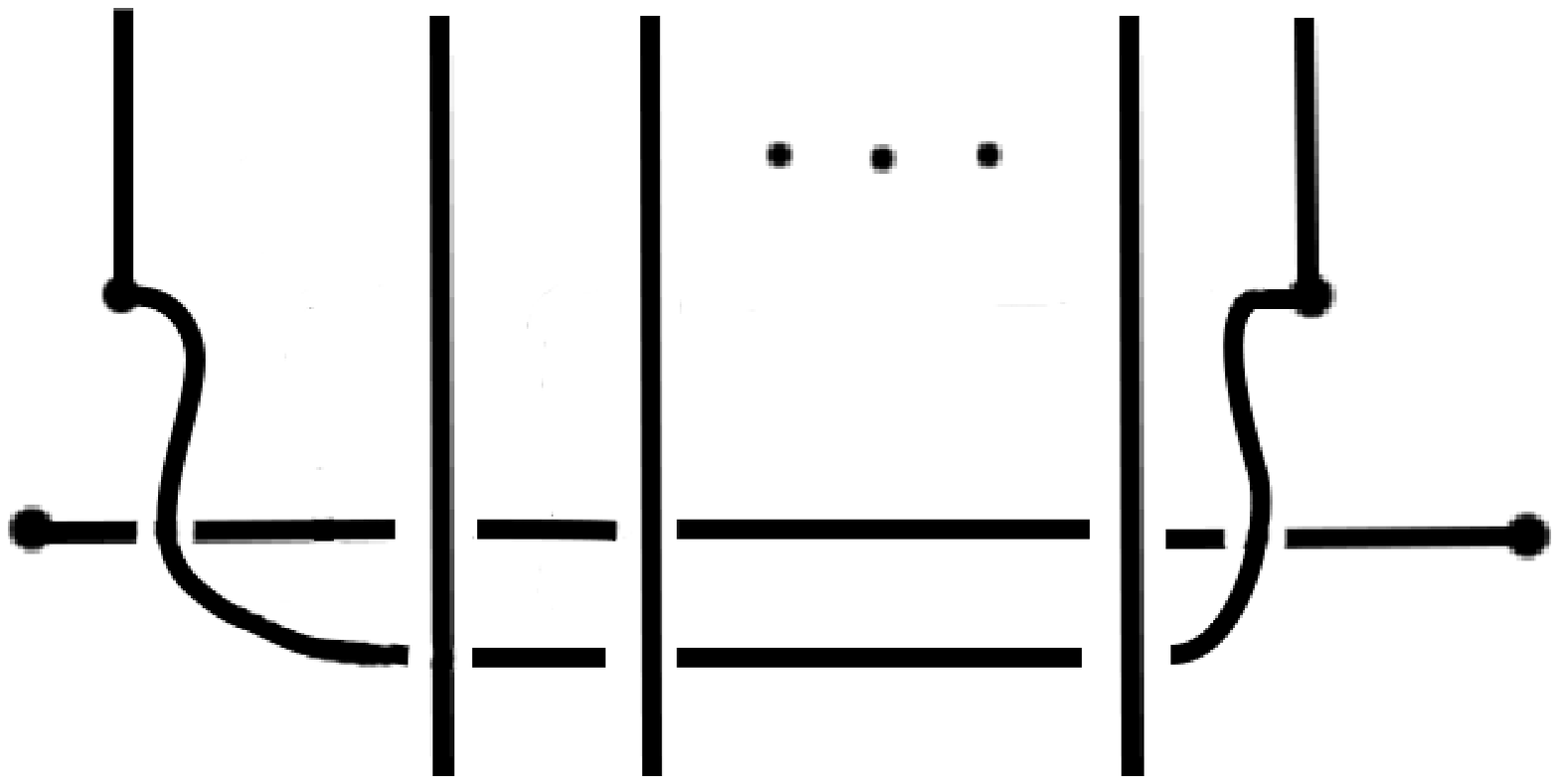}
\vspace{-.2in}
\includegraphics[trim = 0mm 20mm 0mm 0mm, clip,height=1.2in]{arrow.eps}\includegraphics[height=1.3in]{exch1.eps}
\end{center}
\caption{ \small{Factoring an exchange move through a type II and multiple type III Reidemeister moves.}}
\label{factor}
\end{figure}

It is straightforward to show that the maximum number of crossings that may occur in an arc-presentation with complexity less than or equal to $M$ is bounded above by $(M-2)^2$. If we translate an arc--presentation sequence of moves in a canonical fashion into a sequence of Reidemeister moves to unknot our unknot, many knot diagrams in the Reidemeister sequence will be arc--presentations and, as such, will have fewer than $(M-2)^2$ crossings. Furthermore, diagrams in this sequence that are not arc--presentations have no more crossings than their arc--presentation relatives. Thus, there exists a sequence of Reidemeister moves that unknots our original diagram $K$ that does not increase the crossing number to more than $(M-2)^2.$
\end{proof}

\section{Bounds on Reidemeister Moves Needed to Simplify the Unknot}

To find our upper bound on the number of Reidemeister moves, we must first specify an upper bound on the number $m$ of exchange and destabilization moves required to trivialize an arc--presentation. This bound will depend on the complexity $c(L)=n$ of the arc--diagram $L$. We also must provide an upper bound on the number of Reidemeister moves required for a destabilization or exchange move.

In~\cite{dynnikov}, Dynnikov provides the following bounds on the number of combinatorially distinct arc--presentations of complexity $n$. 

\begin{proposition} Let $N(n)$ denote the number of combinatorially distinct arc-- presentations of complexity $n$. Then the following inequality holds.
$$N(n)\leq \frac{1}{2}n [(n-1)!]^2$$
\end{proposition}

\begin{proof}
Suppose we want to create an arc--presentation on the $n\times n$ integer lattice. Let us choose a starting point in the lattice. There are $\frac{n}{2}=\frac{n^2}{2n}$ ways to choose this point since there are $n^2$ lattice points, $2n$ of which lie on a given diagram. From this point, we create a vertical arc ending at another point in the integer lattice. There are $n-1$ choices for this endpoint. From our new point, we want to create a horizontal arc with endpoint in the lattice. There are $n-1$ choices for this endpoint as well. Next, we make another vertical arc, choosing one of the $n-2$ possible endpoints. (There are only $n-2$ choices since no two arcs in the diagram should be colinear.) Similarly, we have $n-2$ choices for the endpoint of our next horizontal arc. Continuing in this fashion, we see that the number of distinct choices we must make is $[(n-1)!]^2$.
Multiplying this quantity by $\frac{n}{2}$ to account for the initial choice of starting point, we get $\frac{1}{2}n[(n-1)!]^2$
\end{proof}

Using this count on the number of distinct arc--presentations of a given size, we can find a bound (albeit a large one) on the number of arc--presentation moves we need. This is simply by virtue of the fact that any reasonable sequence of moves will contain mutually distinct arc--presentations that don't exceed the complexity of the original, and there are a limited number of such diagrams.

\begin{lemma}\label{size}
The number of terms, $m$, in the monotonic simplification of arc--presentation $L$ with $c(L)=n$ is bounded above by $\sum_{i=2}^n\frac{1}{2}i[(i-1)!]^2$.
\end{lemma}
\begin{proof}
Suppose that an arc--presentation $L$ has complexity $n$. Since each $L_k$ from Theorem~\ref{triv} is combinatorially distinct from any other $L_j$ with $k\neq j$, we know that the number $m$ of arc--presentations in the sequence must be at most $\sum_{i=2}^nN(i)$ which is no greater than $\sum_{i=2}^n\frac{1}{2}i[(i-1)!]^2$. 
\end{proof}

We should note that, if we start with an arc--presentation of the unknot, every arc--presentation in our simplification sequence must be a diagram of the unknot. As $n$ gets larger, we recognize that far fewer arc--presentations of complexity $n$ are unknots. Thus, in practice, $m$ will be much lower than the upper bound provided here. The authors would be interested to know what the probability is that an arc--presentation of complexity $n$ is the unknot. Using this probability, we could tighten the upper bound we found above.

We return now to our second question: how many Reidemeister moves does it take to make an arc--presentation move?

\begin{lemma}\label{reid}
No more than $n-2$ Reidemeister moves are required to perform an exchange or destabilization move on an arc--presentation $L$ with complexity $c(L)=n$.
\end{lemma}
\begin{proof}
Clearly, a destabilization move requires at most one Reidemeister move, a type I move. Now consider the first exchange move pictured in Figure~\ref{exch}. Let $d$ be the number of vertical strands intersecting both of the horizontal strands to be switched. Then the move requires $d$ type III moves and one type II move. Thus, the exchange move requires $d+1$ Reidemeister moves. We note that $d<a$, if $a$ is the length of the shorter horizontal arc. But $a$ cannot be greater than $n-2$, so the number of Reidemeister moves required is less than or equal to $n-2$. Similarly, the second exchange move pictured above requires $d$ type III moves but no type II moves. Thus, both pictured exchange moves require no more than $n-2$ Reidemeister moves. We note that other versions of the exchange moves (where the horizontal arcs lie in distinct halves of the arc--presentation) require no Reidemeister moves. 
\end{proof}

For the finale, we put our two results together.

\begin{theorem}
Suppose $K$ is a diagram (in Morse form) of the unknot with crossing number $cr(K)$ and number of maxima $b(K)$. Let $M=2b(K)+cr(K)$. Then the number of Reidemeister moves required to unknot $K$ is less than or equal to $$\sum_{i=2}^M\frac{1}{2}i[(i-1)!]^2(M-2).$$
\end{theorem}

\begin{proof}
Suppose the arc--presentation $L_K$ of our knot diagram $K$ has complexity $c(L_K)=n$. Then at most $m (n-2)$ Reidemeister moves are required to produce the trivial (complexity 2) arc--presentation, where $m$ is the number of moves in the monotonic simplification of $L_K$. By our lemma, this quantity is bounded above by $$\sum_{i=2}^n\frac{1}{2}i[(i-1)!]^2(n-2).$$ But we showed that $n\leq 2b(K)+cr(K)=M$, thus the number of Reidemeister moves required to unknot $K$ is less than or equal to $$\sum_{i=2}^M\frac{1}{2}i[(i-1)!]^2(M-2).$$
\end{proof}

\section{A Detour: Bounds for Untangling Links}\label{links}

To illustrate that similar questions may be extended to families of knots and links beyond the unknot, we take a short detour to the world of non-trivial knots and links. In keeping with our theme, we make use of the work of Dynnikov. He proved two other results regarding the simplification of certain link diagrams~\cite{dynnikov}. In a fashion analogous to the previous section, we may use Dynnikov's results to bound the number of Reidemeister moves and the number of crossings needed to simplify certain types link diagrams. Before we state these theorems, however, let us clearly define our terms. 

\begin{figure}[h]
\begin{center}
\includegraphics[height=1.3in]{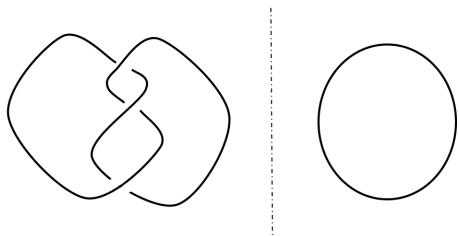}
\end{center}
\caption{ \small{A split link.}}\label{split}
\end{figure}

\begin{figure}[h]
\begin{center}
\includegraphics[height=1.5in]{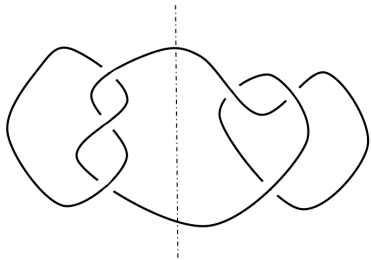}
\end{center}
\caption{ \small{A composite knot.}}\label{composite}
\end{figure}

\begin{definition}
A link diagram $L$ is said to be \emph{split} if there is a line not intersecting $L$ such that there are components of the diagram lying on both sides of the line. A link (or knot) diagram $L$ is \emph{composite} if it can be viewed as a connect sum of two nontrivial links, i.e. if there is a line intersecting the link at two points such that the tangles on either side of the line are non-trivial. In general, a link is said to be split or composite if there exists a diagram of the link that is split or composite. Figures ~\ref{split} and~\ref{composite} give examples illustrating these definitions.
\end{definition}

Let's review the pertinent results from~\cite{dynnikov}.

\begin{theorem}[Dynnikov]
If $L$ is an arc--presentation of a split link, then there exists a finite sequence of exchange and destabilization moves $$L\rightarrow L_1\rightarrow L_2\rightarrow
\cdots\rightarrow L_m$$ such than $L_m$ is split.\end{theorem}

\begin{theorem}[Dynnikov]
If $L$ is an arc--presentation of a non-split composite link, then there exists a finite sequence of exchange and destabilization moves $$L\rightarrow L_1\rightarrow L_2\rightarrow
\cdots\rightarrow L_m$$ such than $L_m$ is composite.\end{theorem}

We note that the statements of Lemmas ~\ref{arclemma},~\ref{size} and~\ref{reid} hold for arbitrary links as well as diagrams of the unknot. Thus, the following result is an immediate consequence of the previous theorems.

\begin{theorem}
Suppose $L$ is a diagram (in Morse form) of a split (resp. non-split composite) link with crossing number $cr(L)$ and number of maxima $b(L)$. Let $M=2b(L)+cr(L)$. Then the number of Reidemeister moves required to transform $L$ into a split (resp. composite) diagram is less than or equal to $$\sum_{i=2}^M\frac{1}{2}i[(i-1)!]^2(M-2).$$
\end{theorem}

Similarly, we have the following extension of our results regarding maximum crossing numbers in a simplifying Reidemeister sequence.

\begin{theorem}
Suppose $L$ is a diagram (in Morse form) of a split (resp. non-split composite) link with crossing number $cr(L)$ and number of maxima $b(L)$. Then for every $i$, the crossing number $cr(L_i)$ is no more than $(M-2)^2$ where $M=2b(L)+cr(L)$ and $L=L_0, L_1, L_2, ..., L_N$ is a sequence of link diagrams such that $L_{i+1}$ is obtained from $L_i$ by a single Reidemeister move, $L_N$ is split (resp. composite).
\end{theorem}


\section{Hard Unknots}

We have provided several upper bounds regarding the complexity of the Reidemeister sequence required to simplify an unknot. The bound that Dynnikov's work helps us obtain for the number of Reidemeister moves required to unknot an unknot is superexponential. Using a different technique, Hass and Lagarias were able to find a bound that is exponential in the crossing number of the diagram~\cite{hl}. They use the same technique to find an exponential bound for the number of crossings required for unknotting. For bounds of this second sort, the one presented here is a comparatively sharper estimate. 

Regarding lower bounds, it was recently shown in ~\cite{hn} that there are unknot diagrams for which the number of Reidemeister moves required for unknotting is quadratic in the crossing number of the initial diagram. In~\cite{hayashi}, similar quadratic lower bounds are given for links. On the other hand, little is known about how many additional crossings an unknot diagram might require in order to become unknotted. While the upper bound on the number of crossings needed in a Reidemeister sequence is merely quadratic in the crossing number of the initial unknot diagram, it nonetheless seems likely that this bound is far from being tight. 

Let us return to our friend, the Culprit. This famous hard unknot diagram was originally discovered by Ken Millett and introduced in~\cite{Culprit}. Recall that hard unknots are difficult to unknot by virtue of the fact that no simplifying type I or type II Reidemeister moves and no type III moves are available. In Figure~\ref{culprit}, we picture a Morse diagram of the Culprit, its corresponding rectangular diagram, and its arc-presentation obtained by rotating crossings where the over-strand was horizontal. Note that we need not specify crossing information in the arc--presentation, for it is assumed that all vertical lines pass over horizontal lines.

\begin{figure}[h]
\begin{center}
\includegraphics[height=1.5in]{Culprit.eps}\includegraphics[height=1.5in]{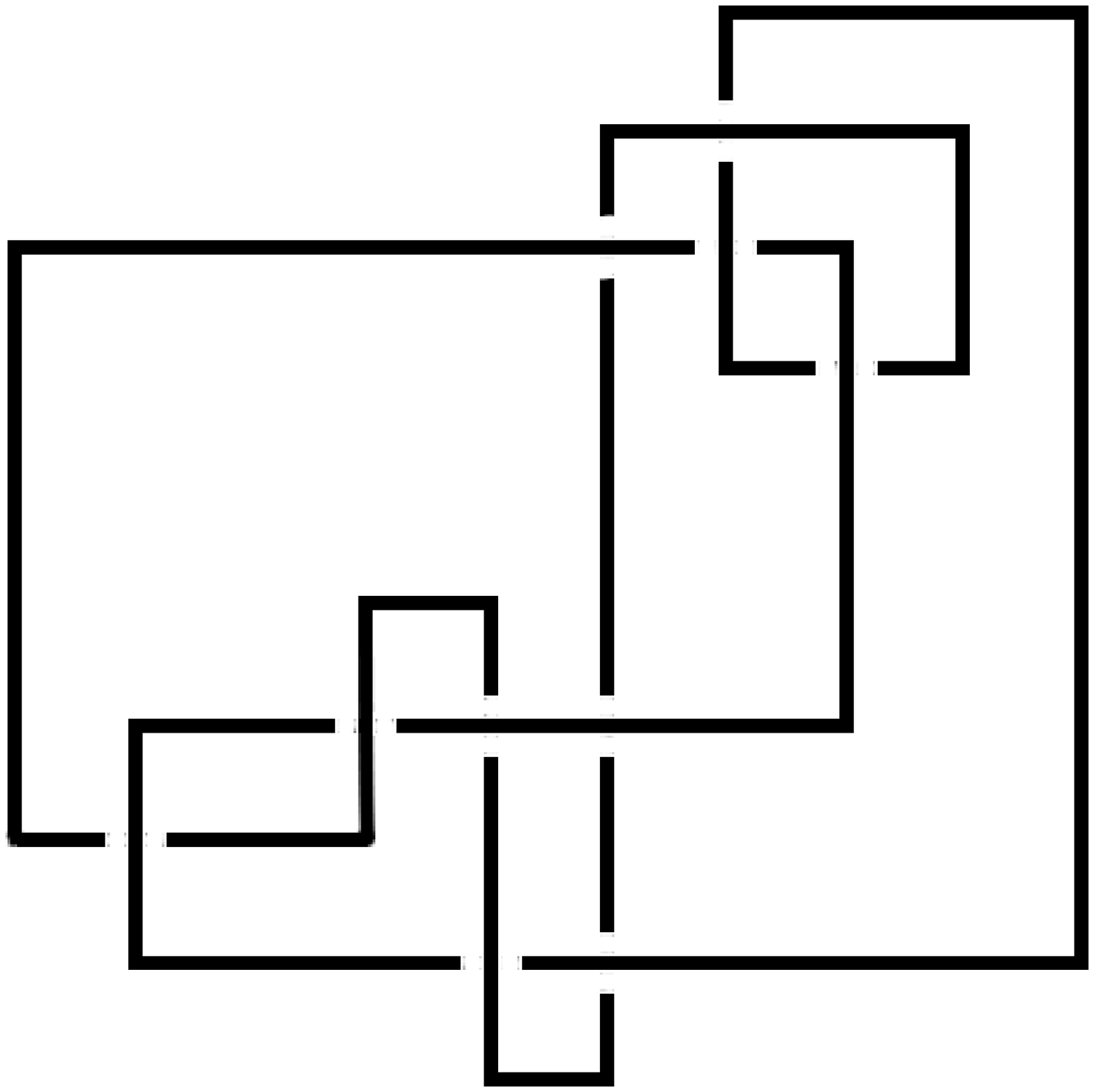}\includegraphics[height=1.5in]{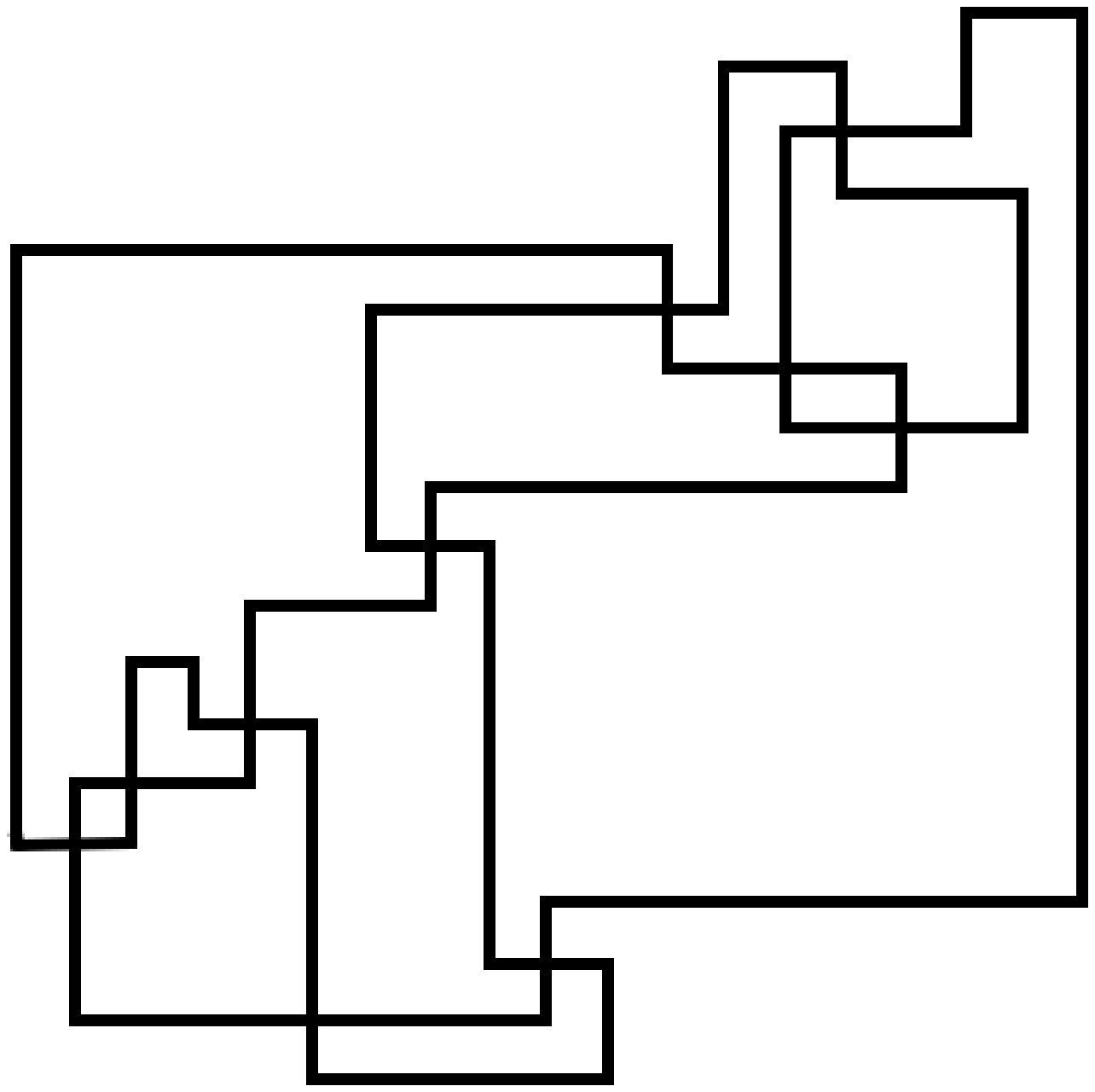}
\end{center}
\caption{ \small{The Culprit with its rectangular diagram and arc--presentation.}}\label{culprit}
\end{figure}

We saw that the Culprit may be unknotted with ten Reidemeister moves. (See also ~\cite{KL}.) The maximum crossing number of all diagrams in the given Reidemeister sequence is 12, two more than the number of crossings in the Culprit. On the other hand, we can compute our upper bound on the number of crossings required for unknotting as follows. Since the crossing number $cr(K)=10$ and the number of maxima in the diagram is $b(K)=5$, we see that $M = cr(K) + 2b(K)=20$. Thus, our bound is $(M-2)^2=18^2=324.$ 

We can also use $M$ to find our bound for the number of Reidemeister moves required to unknot the Culprit. $$\sum_{i=2}^M\frac{1}{2}i[(i-1)!]^2(M-2)=9\sum_{i=2}^{20}i[(i-1)!]^2.$$ The largest term in this expression is roughly $10^{35}$, unfortunately quite a bit larger than ten.

We challenge the reader to find examples where the maximum crossing number is closer to our bound and where the number of needed Reidemeister moves is large in comparison to the number of crossings in the original diagram.

\section{Conclusions}

We've considered the phenomenon that it may be quite hard to unknot a trivial knot and have provided several upper and lower bounds on the number of Reidemeister moves and the number of crossings needed to do the job. Known hard unknots like the Culprit and examples from~\cite{hn} illustrate that unknotting can be tricky, but not as tricky as the upper bounds that are known would have us believe. To answer the questions we've posed, there is much more to be done.

\section{Acknowledgements}

The authors would like to thank Jeffrey Lagarias and John Sullivan for their valuable comments.

\bibliographystyle{abbrv}
\bibliography{ReidBoundExposition3}

\end{document}